\documentclass[12pt]{amsart}

\usepackage{amsmath,amstext,amssymb,amsopn,amsthm}
\usepackage{url,verbatim}
\usepackage{mathtools}
\usepackage{enumerate}

\usepackage{color,graphicx}

\usepackage[margin=30mm]{geometry}
\usepackage{eucal,mathrsfs,dsfont}
\usepackage[section]{placeins}

\allowdisplaybreaks

\newtheorem{theorem}{Theorem}[section]
\newtheorem{corollary}[theorem]{Corollary}
\newtheorem{lemma}[theorem]{Lemma}
\newtheorem{proposition}[theorem]{Proposition}

\theoremstyle{definition}

\newtheorem{remark}[theorem]{Remark}

\theoremstyle{remark}

\numberwithin{equation}{section}

\newcommand{\calT}{\mathcal{T}}

\newcommand{\calS}{\mathcal{S}}

\newcommand{\R}{\mathds{R}}

\newcommand{\ol}{\overline}

\usepackage{alltt}  %

\RequirePackage[colorlinks,citecolor=blue,urlcolor=blue]{hyperref}

\usepackage{amsmath,amstext,amssymb,amsopn,amsthm}
\usepackage{url,verbatim}
\usepackage{mathtools}
\usepackage{enumerate}

\usepackage{color,graphicx}
\usepackage{subfig}
\usepackage{bigints}
\usepackage{bm}

\usepackage{eucal,mathrsfs,dsfont}

\allowdisplaybreaks

\numberwithin{equation}{section}

\theoremstyle{definition}

\newcommand{\dsP}{\mathds{P}}

\newcommand{\bX}{\mathbf{x}}

\newcommand{\bW}{\mathbf{W}}

\newcommand{\bC}{\mathbf{C}}
\newcommand{\bV}{\mathbf{V}}
\newcommand{\bv}{\mathbf{v}}

\newcommand{\bM}{\mathbf{M}}

\title[Collision location]{Collision location for hard spheres\\ in stationary regime}
\author{Krzysztof Burdzy and Shuntao Chen}

\address{Department of Mathematics, Box 354350, University of Washington, Seattle, WA 98195}
\email{burdzy@uw.edu}
\email{shuntao1994@gmail.com}

\thanks{Research supported in part by Simons Foundation Grant 928958. }

\begin{document}

\begin{abstract}
Consider two balls with radius $r>0$ whose centers are at a distance $2$, positioned symmetrically with respect to the origin in $\R^d$.
Suppose that the initial velocities are independent standard normal vectors. 
When $r\to0$, the collision probability goes to 0 as $r^{d-1}$, and the asymptotic collision location
distribution is a (defective) $t$-distribution. This distribution is rotationally symmetric about the origin
for no apparent reason.
\end{abstract}

\maketitle

\section{Introduction}\label{intro}

Physics principles known as ``equidistribution of energy'' and  ``microcanonical ensemble formula'' (see \cite[Chaps. 2-3]{BP}) suggest that velocities of gas molecules in the air are i.i.d. rotationally invariant normal vectors.
Given this assumption, we will find the collision location for two ``randomly'' chosen molecules.
If the gas is dilute, i.e., the distance between a molecule and its nearest neighbor is large relative to its diameter, the probability of a collision of two given molecules is very small. Nevertheless, the probability is non-zero, assuming strictly positive radii $r$ of the molecules. We will 
find asymptotic formulas for the probability and location of a collision of two molecules, assuming a fixed initial distance, as $r\to0$.

We will assume that two molecules are represented by balls initially positioned symmetrically with respect to the origin in $\R^d$. We will suppose that their initial velocities are independent standard normal vectors. 
The cross-section of a ball has a volume of order $r^{d-1}$, so heuristically, the collision probability should be of order $r^{d-1}$. Not surprisingly, this is the case; we will give a rigorous and exact formula for the collision probability. 

To our surprise, we discovered that as $r\to0$,
the asymptotic collision location distribution is a multi-dimensional $t$-distribution for $d\geq 2$. 
This distribution is rotation invariant, i.e., invariant under the action of the orthogonal group $O(d)$, for no apparent reason. The model is not invariant under $O(d)$. It turns out that the collision location is rotation invariant if the moving objects are strictly convex, the union of the two objects is centrally symmetric with respect to the origin at the initial time, and the initial velocities are independent standard normal vectors.

This paper is based on a chapter in the second author's Ph.D. thesis \cite{thesis}. 

\section{Model and main results}\label{sec.collision.1}

Suppose that $d\geq 1$, $\bM_1,\bM_2\subset \R^d$ are closed, $\bM_1\cap \bM_2=\emptyset$, $\bM_1\cup \bM_2$ is centrally symmetric, i.e. $x\in \bM_1\cup \bM_2$ if an only if $-x\in \bM_1\cup \bM_2$, and $\bM_1 $ and $ \bM_2$ are strictly convex. 

We do not assume that $\bM_1 $ and $ \bM_2$  are bounded. Hence, for $d=2$, we could take $\bM_1=\{(x,y): y\geq x^2+1\}$, and make $\bM_2$ the symmetric set.

Suppose that $\bM_1$ and $\bM_2$ are moving with random velocities $\bV_1$ and $\bV_2$. Let $\bM_1(s)=\bM_1 + s\bV_1$ and $\bM_2(s)=\bM_2+s\bV_2$.
Let the collision time and location be defined by $T=\inf\{t\geq 0: \bM_1(t)\cap \bM_2(t)\ne\emptyset\} $
and $\{\bC\}=\bM_1(T)\cap \bM_2(T)$. Note that $\bM_1(T)\cap \bM_2(T)$ contains only one point by the assumption of strict convexity of $\bM_1 $ and $ \bM_2$.
We will write $T<\infty$ to indicate that there is a collision.

\begin{theorem}\label{j2.1}
If the distribution of $(\bV_1, \bV_2)/|(\bV_1, \bV_2)|$ is uniform on the unit sphere in $\R^{2d}$, then the conditional distribution of $\bC$ given $T<\infty$ is rotation invariant, i.e., invariant under the action of the orthogonal group $O(d)$.
\end{theorem}

This gives us immediately the following corollary.

\begin{corollary}\label{j2.2}
    If $\bV_1$ and $ \bV_2$ are independent standard $d$-dimensional normal $\mathcal{N}\left(\mathbf{0}, I_d\right)$ random vectors
then the conditional distribution of $\bC$ given $T<\infty$ is rotation invariant.
\end{corollary}

\begin{remark}
Unfortunately, Corollary \ref{j2.2} seems to contain the only interesting explicit example to which Theorem \ref{j2.1} applies. More precisely, if $\bV_1$ and $ \bV_2$ are independent and the distribution of $(\bV_1, \bV_2)$ is rotation invariant then $\bV_1$ and $ \bV_2$ are independent standard $d$-dimensional normal $\mathcal{N}\left(\mathbf{0}, I_d\right)$, or a constant multiple of such random vectors. We sketch a proof of this claim kindly provided by W\l odek Bryc.
By rotation invariance, $(\bV_1+ \bV_2)/\sqrt{2}$ and $(\bV_1-\bV_2)/\sqrt{2}$ are independent and have the same distributions as $\bV_1$. 
The Bernstein Theorem (see \cite[Thm.~5.1.1]{Bryc}) implies that for any linear functional $A$, random variables $A(\bV_1)$ and $A(\bV_1)$ are normal. It follows that $\bV_1$ and $ \bV_2$ are multidimensional normal. By rotation invariance, all components of $\bV_1$ and $ \bV_2$ have the same variance, and all pairs of components have the same covariance. The off-diagonal covariances are all zero because  $\bV_1$ and $ \bV_2$ are independent. This completes the proof.
\end{remark}

In the rest of this section, we will assume that

(i) $\bM_1 $ and $ \bM_2$ are ``molecules'' in the shape of balls with radii $r\in(0,1)$ and centers
$\bX_1 = (-1, 0, \cdots, 0)$ and $ \bX_2 = (1, 0, \cdots, 0)$, and

(ii) $\bV_1$ and $ \bV_2$ are independent standard $d$-dimensional normal $\mathcal{N}\left(\mathbf{0}, I_d\right)$ random vectors.

Our first result is a simple explicit criterion for a collision.

\begin{proposition}\label{collision-criterion}
If $d\geq 1$ and the initial velocities of $\bM_1$ and $\bM_2$ are non-random vectors $\bv_1$ and $\bv_2$
then the balls  will collide if and only if
$$\cos \beta \leq -\sqrt{1-r^2},$$
where $\beta$ is the angle between the vectors $\bX_1-\bX_2$ and $\bv_1-\bv_2$.
\end{proposition}

Let $p_{r,d}$ denote the probability that  $\bM_1$ and $\bM_2$  collide. 
We have $p_{r,1} = 1/2$ because, in $d=1$, there will be a collision if and only if $\bV_1>\bV_2$.

\begin{theorem}\label{collision-prob}
For $d\geq 2$,
\begin{align}\label{m4.1}
    p_{r,d}&= \frac{1}{2}F\left(\frac{1}{d-1}\cdot\frac{r^2}{1-r^2}; d-1, 1\right)\\
&= \frac{1}{2}\int_{0}^{\frac{1}{d-1}\cdot \frac{r^2}{1-r^2}} \frac{\Gamma\left(\frac{d}{2} \right)}{\Gamma\left(\frac{d-1}{2}\right)\sqrt{\pi}}\left(d-1\right)^{\frac{d-1}{2}}x^{\frac{d-3}{2}}\left(1+\left(d-1\right)x\right)^{-\frac{d}{2}} dx,\label{m18.1}
\end{align}
where $F(\,\cdot\,; d-1, 1)$ is the CDF of the $F$-distribution with $d-1$ and $1$ degrees of freedom.
It follows that
\begin{align}\label{m4.2}
    p_{r,d} = \frac{1}{d-1}\cdot\frac{\Gamma\left(\frac{d}{2}\right)}{\Gamma\left(\frac{d-1}{2}\right)\sqrt{\pi}}r^{d-1}+o(r^{d-1}).
\end{align}
\end{theorem}

\begin{remark}\label{collition-prob-approx-special}
In this and the next remark, 
we will discuss the special cases $d=2,3$ because these dimensions are physically relevant, and the corresponding formulas are simple.
Substituting $d=2,3$ into \eqref{m4.2}, we  get
$p_{r,2} = \frac{1}{\pi}r + o(r)$ and $ p_{r,3} = \frac{1}{4}r^2+o(r^2)$.
Actually, we can compute exact values of $p_{r,2}$ and $p_{r,3}$.
If we substitute $d=2, 3$ in \eqref{m18.1}, we get
\begin{align}\label{m18.3}
p_{r,2} &= \frac{1}{2}\int_{0}^{\frac{r^2}{1-r^2}} \frac{1}{\pi}\cdot\frac{1}{\sqrt{x}\left(1+x\right)}dx=\frac{1}{\pi}\arctan \sqrt{x}\bigg{|}_{0}^{\frac{r^2}{1-r^2}}\\
&=\frac{1}{\pi}\arctan \sqrt{\frac{r^2}{1-r^2}}=\frac{r}{\pi} + O(r^2),\label{m18.4}\\
p_{r,3} &= \frac{1}{2}\int_{0}^{\frac{r^2}{2(1-r^2)}} (1+2x)^{-\frac{3}{2}}dx= -\frac{1}{2}(1+2x)^{-\frac{1}{2}}\bigg{|}_{0}^{\frac{r^2}{2(1-r^2)}}\label{m18.5}\\
&=\frac{1}{2}-\frac{1}{2}\left(1+\frac{r^2}{1-r^2}\right)^{-\frac{1}{2}}=\frac{1-\sqrt{1-r^2}}{2}=\frac{1}{4}r^2+O(r^4).\label{m18.6}
\end{align}

\end{remark}

Next, we turn to the collision location.

\begin{theorem}\label{density-1-dim}
For $d=1$, 
\begin{equation}\label{formula-1-dim}
\dsP(\bC\leq x) = \int_{-\infty}^{x} \frac{1-r}{2\pi\left(u^2 + \left(1-r\right)^2\right)} du.
\end{equation}

\end{theorem}

\begin{theorem}\label{thm_high_dim}
 Suppose $d\geq2$ and  $A\subset \mathbb{R}^d$ is a bounded Borel set with a positive distance from the origin.  Then
\begin{equation}\label{m18.2}
\lim_{r\to 0}
\dsP(\bC\in A) r^{1-d} =
\frac{ \pi ^{-(d+1)/2}  \Gamma (d) }
{2\Gamma(d/2+1/2) }
 \int_{A}  \frac{1}{\left(1+|x|^2 \right)^d} d x .
\end{equation}

\end{theorem}

\begin{remark}

(i) The distribution in \eqref{m18.2} is a special case of the multidimensional $t$-distribution; more precisely, it is a defective version of the multidimensional $t$-distribution.
The formula in \eqref{m18.2} agrees with \eqref{formula-1-dim}, except that the latter is an exact formula, unlike the asymptotic formula in \eqref{m18.2}.

(ii) The values of the coefficient $(1/2) \pi ^{-(d+1)/2}  \Gamma (d) /\Gamma(d/2+1/2)$ in \eqref{m18.2} for $d=2,\dots,11$  are
\begin{align}\label{m18.7}
\frac{1}{\pi ^2},\frac{1}{\pi ^2},\frac{4}{\pi ^3},\frac{6}{\pi ^3},\frac{32}{\pi ^4},\frac{60}{\pi ^4},\frac{384}{\pi ^5},\frac{840}{\pi ^5},\frac{6144}{\pi ^6},\frac{15120}{\pi ^6}.
\end{align}

(iii) We will show that the formula \eqref{m18.2} agrees with \eqref{m4.2} and \eqref{m18.3}-\eqref{m18.6} for $d=2,3$. 
Note that according to \eqref{m18.7}, the value of the coefficient $(1/2) \pi ^{-(d+1)/2}  \Gamma (d) /\Gamma(d/2+1/2)$  in \eqref{m18.2} is $1/\pi^2$ for $d=2,3$.
We compute the integral in \eqref{m18.2} for $d=2$ using polar coordinates, 
\begin{align*}
 \frac{r}{\pi^2} \iint_{\mathbb{R}^2} \frac{1}{(1+x^2+y^2)^2} dxdy 
 = \frac{r}{\pi^2}\int_0^\infty \int_{0}^{2\pi} \frac{\rho}{(1+\rho^2)^2}d\theta d\rho = \frac{r}{\pi}.
\end{align*}

In the case $d=3$, we
apply the following polar coordinates in $\mathbb{R}^3$,
\begin{equation*}
x = \rho\cos\theta_1, \quad y = \rho \sin\theta_1\cos\theta_2, \quad z = \rho\sin\theta_1\sin\theta_2, 
\quad \rho>0,\; \theta_1\in [0, \pi),\; \theta_2\in [0, 2\pi),
\end{equation*}
and substitution $u=\frac{1}{1+\rho^2}$
to obtain
\begin{align*}
& \frac{r^2}{\pi^2} \iint_{\mathbb{R}^3} \frac{1}{(1+x^2+y^2+z^2)^3} dxdydz 
= \frac{r^2}{\pi^2}\int_0^\infty \frac{\rho^2}{(1+\rho^2)^2} d\rho \int_{0}^{\pi} \sin \theta_1 d\theta_1\int_0^{2\pi} d\theta_2\\
&= \frac{r^2}{\pi^2}\cdot 2\cdot 2\pi\cdot \int_0^1 \frac{1}{2}\sqrt{u}\sqrt{1-u} du 
= \frac{r^2}{\pi^2}\cdot 2\cdot 2\pi\cdot B\left(\frac{3}{2}, \frac{3}{2} \right) 
= \frac{r^2}{4},
\end{align*}
where $B\left(\frac{3}{2}, \frac{3}{2} \right)$ denotes the Beta function.
As expected, \eqref{m18.2} agrees with \eqref{m4.2} and \eqref{m18.3}-\eqref{m18.6} in the cases $d=2,3$.
\end{remark}

\section{General moving sets}
\label{sec:higher:new}

A collision will occur if and only if the trajectory of $t\to(\bV_1,\bV_2)t$ passes through the set
\begin{align}\label{m21.1}
A:= \{(y +  z ,y- z)\in\R^{2d}: y\in \R^d,z\in \bM_2 \}. 
\end{align}
If $(y + z ,y- z)$ is the first point in $A$ encountered by $(\bV_1,\bV_2)t$, then $\bC=y$.
Let $A_+ =  A\cup\{\mathbf{0}\}$.

For an element $g$ of the orthogonal group $O(d)$, 
let $\calT_g : A_+ \to A_+$ be defined by $\calT_g(\mathbf{0})=\mathbf{0}$, and
for $(y +  z ,y- z)\in A$,
\begin{align*}
    \calT_g & (y +  z ,y- z) = (g(y) +  z ,g(y)- z).
\end{align*}
\begin{lemma}\label{m16.1}
     For every $g\in O(d)$, $\calT_g$ is an isometry on $A_+$.
\end{lemma}

\begin{proof}
    
If  $y_1,y_2\in\R^d$ and $z_1,z_2\in \bM_2$  then
\begin{align*}
&|(y_1 +  z_1 ,y_1- z_1)- (y_2 +  z_2 ,y_2- z_2)|^2\\
&= |((y_1-y_2) +  (z_1-z_2) ,(y_1-y_2) -  (z_1-z_2))|^2\\
&= |y_1-y_2|^2 + 2  (y_1-y_2) \cdot (z_1-z_2) +  |z_1-z_2|^2\\
&\qquad + |y_1-y_2|^2 - 2  (y_1-y_2) \cdot (z_1-z_2) +  |z_1-z_2|^2\\
&= 2 |y_1-y_2|^2 + 2   |z_1-z_2|^2.
\end{align*}
Every $g\in O(d)$  is an isometry, so $|g(y_1)-g(y_2)| = |y_1-y_2|$ and, therefore,
\begin{align*}
&|\calT_g(y_1 +  z_1 ,y_1- z_1)- \calT_g(y_2 +  z_2 ,y_2- z_2)|^2\\
&=|(g(y_1) +  z_1 ,g(y_1)- z_1)- (g(y_2) +  z_2 ,g(y_2)- z_2)|^2\\
&= 2 |g(y_1)-g(y_2)|^2 + 2   |z_1-z_2|^2 = 2 |y_1-y_2|^2 + 2   |z_1-z_2|^2\\
&=|(y_1 +  z_1 ,y_1- z_1)- (y_2 +  z_2 ,y_2- z_2)|^2.
\end{align*}
This shows that $\calT_g$ is an isometry on $A$.

If  $y\in\R^d$ and $z\in \bM_2$  then
\begin{align*}
&|(y +  z ,y- z)- \mathbf{0}|^2=|(y +  z ,y- z)|^2
 = 2|y|^2 +2|z|^2.
\end{align*}
Since $|g(y)|=|y|$,
\begin{align*}
&|\calT_g(y + r z ,y-r z)- \mathbf{0}|^2
= 2|g(y)|^2 +2 |z|^2 = |(y + r z ,y-r z)- \mathbf{0}|^2.
\end{align*}
This proves that $\calT_g$ is an isometry on $A_+$.
\end{proof}

\begin{proof}[Proof of Theorem \ref{j2.1}]
    Suppose $B\subset \R^d$ and let
    \begin{align*}
A_B= \{(y +  z ,y- z)\in\R^{2d}: y\in B,z\in \bM_2\}. 
\end{align*}
The collision point $\bC $ is in $B$ if and only if $(\bV_1,\bV_2) T \in A_B$.
The distribution of $(\bV_1,\bV_2)/|(\bV_1,\bV_2)|$ is uniform on the unit sphere.
By Lemma \ref{m16.1}, for any $g\in O(d)$, the radial projection of $A_B$ on the unit sphere 
has the same surface area as the projection of $\calT_g(A_B)$. Hence, the probability
that $(\bV_1,\bV_2) t$ will intersect $A_B$ is the same as the probability that it will intersect $\calT_g(A_B)$.
It follows that $\dsP(\bC\in B) = \dsP(\bC\in g(B))$.
\end{proof}

\begin{remark}
For $B_1\subset \R^d$ and $B_2 \subset \bM_2$ let
    \begin{align*}
A_{B_1,B_2}= \{(y +  z ,y- z)\in\R^{2d}: y\in B_1,z\in B_2\}. 
\end{align*}
The argument used in the proof of Theorem \ref{j2.1} also shows that 
\begin{align*}
\dsP((\bV_1,\bV_2) T  \in A_{B_1,B_2})
= \dsP((\bV_1,\bV_2) T  \in A_{g(B_1),B_2}),
\end{align*}
for any $B_1\subset \R^d$, $B_2 \subset \bM_2$, and $g\in O(d)$.
\end{remark}

\section{Collision probability}\label{sec-col-prob}

\begin{proof}[Proof of Proposition \ref{collision-criterion}]

We will use $|\,\cdot\,|$ to denote the usual norm in the Euclidean space.

At time $s$, the centers of $\bM_1$ and $\bM_2$ will be $\bX_1+s\bv_1$ and $\bX_2+s\bv_2$, so to have a collision, we need 
$$|\left(\bX_1+s\bv_1\right) -\left(\bX_2+s\bv_2\right) | = 2r,$$
for some $s>0$. This is equivalent to
\begin{align}\notag
&\left[\left(\bX_1-\bX_2\right)+s\left(\bv_1-\bv_2\right)\right]^T\left[\left(\bX_1-\bX_2\right)+s\left(\bv_1-\bv_2\right)\right] = 4r^2,\\
 & |\bv_1-\bv_2|^2 s^2 + 2\left(\bv_1-\bv_2\right)^T\left(\bX_1-\bX_2\right)s+|\bX_1-\bX_2|^2-4r^2=0.
\label{3.1}
\end{align}
To make the above quadratic equation have a root in $(0, +\infty)$, we  need 
$$\Delta = 4\left[\left(\bv_1-\bv_2\right)^T\left(\bX_1-\bX_2\right)\right]^2-4|\bv_1-\bv_2|^2\left(|\bX_1-\bX_2|^2-4r^2\right)\geq 0,$$
which is equivalent to
\begin{align*}
&|\bv_1-\bv_2|^2|\bX_1-\bX_2|^2\cos^2\beta - |\bv_1-\bv_2|^2\left(|\bX_1-\bX_2|^2-4r^2\right)\geq 0,\\
 &\cos^2\beta-1+\frac{4r^2}{|\bX_1-\bX_2|^2}\geq 0.
\end{align*}
Since $|\bX_1-\bX_2| = 2$, the above condition is the same as
\begin{equation}\label{3.2}
\cos^2\beta\geq 1-r^2.
\end{equation}

If $s_1, s_2$ are the two roots of \eqref{3.1},  Vieta's formulas give
\begin{align*}
\begin{cases}
s_1+s_2 = -\frac{2\left(\bv_1-\bv_2\right)^T\left(\bX_1-\bX_2\right)}{|\bv_1-\bv_2|^2} = -\frac{2|\bX_1-\bX_2|}{|\bv_1-\bv_2|}\cos\beta,\\
s_1s_2 = \frac{|\bX_1-\bX_2|^2-4r^2}{|\bv_1-\bv_2|^2}>0.
\end{cases}
\end{align*}
Hence, for \eqref{3.1} to have a positive solution, we need $\cos\beta\leq 0$. 
This and  \eqref{3.2} yield the condition
$\cos\beta \leq -\sqrt{1-r^2}$.
\end{proof}

\begin{remark}
When $r$ is close to $0$, to have a collision, $\cos\beta$ should be close to $-1$. This means $\beta$ should be close to $\pi$, i.e., $\bX_1-\bX_2$ and $\bv_1-\bv_2$ must have almost the opposite directions. Hence, if $\bM_1$ and $\bM_2$ were to collide, $\bv_1-\bv_2$ should have almost the same direction as the positive half of the first axis. 
\end{remark}

\begin{proof}[Proof of Theorem \ref{collision-prob}]
Let
\begin{equation}\label{def-component}
\bV_1 = \left(v_{1,1}, v_{1,2}, \cdots, v_{1,d}\right), \quad\bV_2 = \left(v_{2,1}, v_{2,2},\cdots, v_{2,d}\right).
\end{equation}
Let $\beta$ be the angle between the vectors $\bX_1-\bX_2$ and $\bV_1-\bV_2$.
Given Proposition \ref{collision-criterion}, we only need to show
$$\dsP\left(\cos\beta\leq -\sqrt{1-r^2}\right) = \frac{1}{2}F\left(\frac{1}{d-1}\cdot\frac{r^2}{1-r^2}; d-1, 1\right).$$

Recall that $\bX_1 = (-1, 0, \cdots, 0)$, $ \bX_2 = (1, 0,\cdots, 0)$,
and  $F\left(\,\cdot\,; d-1, 1 \right)$ denotes the CDF of the $F$ distribution
with $d-1$ and 1 degrees of freedom.
We have
\begin{equation}\label{3.5}
\begin{split}
&\dsP\left(\cos\beta\leq -\sqrt{1-r^2}\right)\\
&= \dsP\left(\frac{-(v_{1,1}-v_{2,1})}{|\bV_1-\bV_2|}\leq -\sqrt{1-r^2}\right)\\
&= \dsP\left(v_{1,1}-v_{2,1}\geq \sqrt{1-r^2}|\bV_1-\bV_2| \right)\\
&=\dsP\left(\left(v_{1,1}-v_{2,1}\right)^2\geq (1-r^2)\sum_{i=1}^d \left(v_{1,i}-v_{2,i}\right)^2, v_{1,1}-v_{2,1}\geq 0 \right)\\
&=\dsP\left(r^2\left(v_{1,1}-v_{2,1}\right)^2\geq (1-r^2)\sum_{i=2}^d \left(v_{1,i}-v_{2,i}\right)^2, v_{1,1}-v_{2,1}\geq 0 \right)\\
&=\frac{1}{2}\dsP\left(\frac{\frac{1}{d-1}\sum_{i=2}^d \left(v_{1,i}-v_{2,i}\right)^2}{\left(v_{1,1}-v_{2,1}\right)^2}\leq \frac{1}{d-1}\cdot\frac{r^2}{1-r^2}\right)\\
&=\frac{1}{2}F\left(\frac{1}{d-1}\cdot \frac{r^2}{1-r^2}; d-1, 1 \right).
\end{split}
\end{equation}
To justify the last equality, we
note that $\{v_{1,i}-v_{2,i}\}_{i=1}^d$ are i.i.d. $\mathcal{N}(0,2)$, so the random variables $\sum_{i=2}^d \left(v_{1,i}-v_{2,i}\right)^2 $ and $ \left(v_{1,1}-v_{2,1}\right)^2$ are independent. Moreover, the distribution of $\frac14\sum_{i=2}^d \left(v_{1,i}-v_{2,i}\right)^2 $ is $\chi_2^{d-1}$, while the distribution of  $\frac14\left(v_{1,1}-v_{2,1}\right)^2$ is $\chi_2^1$. Hence, the last equality follows from the definition of the $F$ distribution.
This proves \eqref{m4.1}.

The PDF of $F$ distribution with $d-1$ and $1$ degrees of freedom is:
$$\frac{\Gamma\left(\frac{d}{2} \right)}{\Gamma\left(\frac{d-1}{2}\right)\sqrt{\pi}}\left(d-1\right)^{\frac{d-1}{2}}x^{\frac{d-3}{2}}\left(1+\left(d-1\right)x\right)^{-\frac{d}{2}}, \quad x>0.$$
This and \eqref{3.5} yield \eqref{m18.1}, i.e.,
\begin{equation*}
p_{r,d} = \frac{1}{2}\int_{0}^{\frac{1}{d-1}\cdot \frac{r^2}{1-r^2}} \frac{\Gamma\left(\frac{d}{2} \right)}{\Gamma\left(\frac{d-1}{2}\right)\sqrt{\pi}}\left(d-1\right)^{\frac{d-1}{2}}x^{\frac{d-3}{2}}\left(1+\left(d-1\right)x\right)^{-\frac{d}{2}} dx.
\end{equation*}
By L'H\^ opital's Rule, 
\begin{alignat*}{2}
\lim_{r\rightarrow 0}\;\frac{p_{r,d}}{r^{d-1}} &=\lim_{r\rightarrow 0}\; \frac{1}{2}\cdot \frac{\Gamma\left(\frac{d}{2} \right)}{\Gamma\left(\frac{d-1}{2}\right)\sqrt{\pi}}\left(d-1\right)^{\frac{d-1}{2}}\left(\frac{1}{d-1}\cdot \frac{r^2}{1-r^2}\right)^{\frac{d-3}{2}}\\
& \quad\quad\times \left(1+\frac{r^2}{1-r^2}\right)^{-\frac{d}{2}}\cdot \frac{1}{d-1}\cdot \frac{2r}{(1-r^2)^2}\cdot\frac{1}{(d-1)r^{d-2}}\\
&=  \lim_{r\rightarrow 0}\;  \frac{1}{d-1}\cdot \frac{\Gamma\left(\frac{d}{2} \right)}{\Gamma\left(\frac{d-1}{2}\right)\sqrt{\pi}}\left(1-r^2\right)^{-\frac{1}{2}}
=  \frac{1}{d-1}\cdot \frac{\Gamma\left(\frac{d}{2} \right)}{\Gamma\left(\frac{d-1}{2}\right)\sqrt{\pi}}.
\end{alignat*}
This proves \eqref{m4.2}.
\end{proof}

\section{Collision location in dimension $1$}
\label{sec-coplace-1-dim}

The 1-dimensional case is very easy due to 
the following well-known representation of the Cauchy distribution.

\begin{lemma}\label{cauchy}
Suppose the endpoint of a half-line is a fixed point $(a,b)$ with $b>0$, and its angle relative to the first axis is uniform in $(-\pi,0)$.  Then the first coordinate of the intersection point with the first axis  has the Cauchy distribution with the probability density function
$$f(x) = \frac{1}{\pi}\cdot \frac{b}{\left(x - a\right)^2 + b^2}.$$
\end{lemma}

\begin{proof}[Proof of Theorem \ref{density-1-dim}]
Let $\bW(s) = \left(\bW_1(s), \bW_2(s)\right)$. The halfline $L:=\{\bW(s),s>0\}$ is the trajectory of the centers of the balls. The initial location is $\bW(0) = (-1, 1)$. Note that the ball $\bM_1$ is always to the left side of ball $\bM_2$, so if the collision time $T$ is finite then
$$\bW_2(T) - \bW_1(T) =\left(\bX_2 + T\bV_2\right) - \left(\bX_1 + T\bV_1\right) = 2r,$$
which is equivalent to $L$ intersecting the line $\ell:=\{(x,y): y-x = 2r\}$. 

Since $\bV_1, \bV_2$ are distributed as i.i.d. standard normal, the direction of $\bW(s)$ is uniform on the circle.

The transformation $\calT: (x, y) \mapsto (y+x, y-x)$
takes $\ell$ to the line $\ell_1:=\{(x,y):y=2r\}$ and $\bW(0)$ to the point $(0, 2)$. The distance between $(0, 2)$ and  $\ell_1$ is $2-2r$. Applying Lemma \ref{cauchy}, the fact that trajectory hits the line $\ell$ with probability $1/2$ and $\bW_2(T) = \bW_1(T) +2r$, we get
\begin{align*}
\dsP\left(\bW_1(T)\leq x \right) &= \dsP\left(\bW_1(T)+\bW_2(T)\leq 2x + 2r \right) 
= \frac{1}{2}\int_{-\infty}^{2x+2r} \frac{1}{\pi}\cdot \frac{2-2r}{u^2+ (2-2r)^2} du\\
&= \int_{-\infty}^{x} \frac{1}{2\pi}\cdot \frac{1-r}{\left(u+r\right)^2+\left(1-r\right)^2} du.
\end{align*}
This implies \eqref{formula-1-dim} because  $\bC = \bW_1(T)+r$ if $d=1$.
\end{proof}

\section{Collision location in dimensions $d\geq 2$}
\label{sec:higher2}

\begin{proof}[Proof of Theorem \ref{thm_high_dim}]

We will decompose the motion into the motion of the center of mass and 
the motion relative to the center of mass. Recall \eqref{def-component} and
for $k=1,\dots,d$ let 
\begin{align*}
    \ol v_k &= (v_{1,k}+v_{2,k})/2,\\
    v_k^c &= (v_{1,k}-v_{2,k})/2,\\
    \ol\bV &= (\ol v_1, \dots, \ol v_d),\\
    \bV^c &= (v_1^c, \dots, v_d^c).
\end{align*}
Since the random variables $v_{1,1},\dots , v_{1,d}, v_{2,1},\dots,v_{2,d}$ are i.i.d. standard normal, the vectors $\ol\bV$ and $\bV^c$ are independent mean zero normal. The covariance matrix for each of these vectors is the identity matrix times  $1/2$. 

The collision time $T$ is determined by $\bV^c$ because
\begin{align*}
    T= \inf\{t>0: 0 \in \bM_1 + \bV^c t\}= \inf\{t>0: 0 \in \bM_2 - \bV^c t\}.
\end{align*}
Then $\bC= \ol\bV T$. Conditionally on $T$, the distribution of $\bC$ is rotationally invariant
by independence of $\ol\bV$ and $\bV^c$ and rotation invariance of $\ol \bV$. It follows that the (unconditional) distribution of $\bC$
is rotationally invariant.

For $x\in\R^n$ and $a\geq 0$, let $\calS_n(x,a)=\{y\in\R^n: |y-x|=a\}$. 
For $z\in \calS_d(0,1)$, let 
\begin{align*}
    \rho(z) = \inf\{b>0: 0 \in \bM_1 + z b\}.
\end{align*}
Then 
\begin{align*}
    T= \frac{\rho(\bV^c/|\bV^c|)}{|\bV^c|}.
\end{align*}

Random variables $\bV^c/|\bV^c|$ and $|\bV^c|$ are independent.
Thus, the conditional distribution of $\bC$ given $\{\rho(\bV^c/|\bV^c|)=u\}$
is that of $u \ol \bV /|\bV^c|$. Therefore, 
the conditional distribution of $|\bC|^2$ given $\{\rho(\bV^c/|\bV^c|)=u\}$
is that of $u^2 |\ol \bV|^2 /|\bV^c|^2$. 
The random variables $\sqrt{2} |\ol \bV|^2 $ and $\sqrt{2}|\bV^c|^2$ are $\chi^2_d$
so  $ |\ol \bV|^2 /|\bV^c|^2$ has the $F$-distribution with $d$ and $d$ degrees
of freedom. We have
\begin{align*}
    \dsP(|\bC| < a ) &= \dsP(|\bC|^2 < a^2 )
    = \int \dsP(|\bC|^2 < a^2 \mid \rho(\bV^c/|\bV^c|)=u) \dsP(\rho(\bV^c/|\bV^c|)\in du)\\
    &= \int \dsP(u^2 |\ol \bV|^2 /|\bV^c|^2 < a^2 \mid \rho(\bV^c/|\bV^c|)=u) \dsP(\rho(\bV^c/|\bV^c|)\in du)\\
    &= \int \dsP( |\ol \bV|^2 /|\bV^c|^2 < a^2/u^2 ) \dsP(\rho(\bV^c/|\bV^c|)\in du).
\end{align*}
The density of the $F$-distribution with $d$ and $d$ degrees of freedom is
\begin{align*}
    f(x) = \frac{x^{d/2-1}}{B(d/2,d/2)(1+x)^d}
\end{align*}
so
\begin{align*}
    \dsP( |\ol \bV|^2 /|\bV^c|^2 < a^2/u^2 )
    &= \int _0^{a^2/u^2} \frac{x^{d/2-1}}{B(d/2,d/2)(1+x)^d} dx,\\
    \dsP(|\bC| < a ) &= 
     \int \int _0^{a^2/u^2} \frac{x^{d/2-1}}{B(d/2,d/2)(1+x)^d} dx \dsP(\rho(\bV^c/|\bV^c|)\in du).
\end{align*}

When $r\to0$, $\rho(z)$ converges uniformly to 1 on the set where it is finite.
Hence, $\rho(\bV^c/|\bV^c|)\to 1$ on the event $T<\infty$. 
It follows that 
\begin{align*}
  \lim_{r\to0}  \dsP(|\bC| < a )/\dsP(T<\infty) &= 
      \int _0^{a^2} \frac{x^{d/2-1}}{B(d/2,d/2)(1+x)^d} dx .
\end{align*}
We have
\begin{align*}
    \frac{d}{da} \int _0^{a^2} \frac{x^{d/2-1}}{B(d/2,d/2)(1+x)^d} dx
    = 2a\frac{a^{d-2}}{B(d/2,d/2)(1+a^2)^d}
    = \frac{2a^{d-1}}{B(d/2,d/2)(1+a^2)^d}.
\end{align*}
The above formula gives us the density of the
asymptotic distribution of $|\bC|$ given $T<\infty$. To transform
this formula into the density of $\bC$, we have to  divide it by the surface area of $\calS_d(0,a) $, i.e., $2\pi^{d/2}a^{d-1}/\Gamma(d/2)$.
The density of the
asymptotic distribution of $\bC$ given $T<\infty$ is
\begin{align*}
    \frac{2|x|^{d-1}}{B(d/2,d/2)(1+|x|^2)^d} \cdot \frac{\Gamma(d/2)}{2\pi^{d/2}|x|^{d-1}}
    &= \frac{\Gamma(d/2)}{\pi^{d/2}B(d/2,d/2)}\cdot\frac{1}{(1+|x|^2)^d}\\
    &= \frac{\Gamma(d)}{\pi^{d/2}\Gamma(d/2)}\cdot\frac{1}{(1+|x|^2)^d}.
\end{align*}
at $x\in\R^d$.
 Suppose $d\geq2$ and  $A\subset \mathbb{R}^d$ is a bounded Borel set with a positive distance from the origin.  Then
\begin{equation*}
\lim_{r\to 0}
\dsP(\bC\in A) /\dsP(T<\infty) = \frac{\Gamma(d)}{\pi^{d/2}\Gamma(d/2)}
 \int_{A}  \frac{1}{\left(1+|x|^2 \right)^d} d x .
\end{equation*}
Elementary calculations show that this and \eqref{m4.2} yield \eqref{m18.2}.
\end{proof}

\section{Acknowledgments}

We thank W.~Bryc, J.~Wellner, S.~Steinerberger and J.~Weso\l owski for the most useful advice.

\def\cprime{$'$}


\begin{thebibliography}{1}

\bibitem{Bryc}
W{\l}odzimierz Bryc.
\newblock {\em The normal distribution}, volume 100 of {\em Lecture Notes in Statistics}.
\newblock Springer-Verlag, New York, 1995.
\newblock Characterizations with applications.

\bibitem{thesis}
Shuntao Chen.
\newblock {\em Pinned Balls, Foldings and Particle Collisions}.
\newblock PhD thesis, University of Washington, 2024.

\bibitem{BP}
R.~K. Pathria and Paul~D. Beale.
\newblock {\em Statistical Mechanics}.
\newblock Academic Press, Amsterdam, 2011.

\end{thebibliography}
\end{document}